\documentclass[alphabetic]{amsart}

\usepackage{enumerate, longtable}
\usepackage{amsmath, amscd, amsfonts, amsthm, amssymb, latexsym, comment, stmaryrd, graphicx, dsfont, amsaddr, mathtools, stmaryrd}
\usepackage{xcolor}
\usepackage[all]{xy}
\usepackage{fullpage}
\usepackage{tikz}
\usepackage[
        colorlinks,
        linkcolor=red,  citecolor=blue,
        backref
]{hyperref}

\usepackage{amstext} 
\usepackage{array}
\usepackage{cleveref}

\newtheorem{claim}{Claim}[section]

\newtheorem{conj}[claim]{Conjecture} 
\newtheorem{question}[claim]{Question} 
\newtheorem{thm}{Theorem}
\newtheorem{prop}[claim]{Proposition}

\theoremstyle{definition}

\newtheorem{rem}[claim]{Remark}

\newcommand{\F}{{\mathbb F}}
\newcommand{\Z}{{\mathbb Z}}
\newcommand{\Q}{{\mathbb Q}}

\newcommand{\R}{{\mathbb R}}

\newcommand{\N}{{\mathbb N}}

\newcommand{\Gal}{{\mathrm{Gal}}}
\newcommand{\Aut}{{\mathrm{Aut}}}

\newcommand{\Ram}{{\mathrm{Ram}}}

\newcommand{\disc}{{\mathrm{disc}}}

\numberwithin{equation}{section}

\author{Alexei Entin}
\address{Raymond and Beverly Sackler School of Mathematical Sciences, Tel Aviv University, Tel Aviv 69978, Israel}
\email{aentin@tauex.tau.ac.il}
\title{Automorphism Groups of Finite Extensions of Fields and the Minimal Ramification Problem}

\begin{document}

\begin{abstract} We study the following question: given a global field $F$ and finite group $G$, what is the minimal $r$ such that there exists a finite extension $K/F$ with $\Aut(K/F)\cong G$ that is ramified over exactly $r$ places of $F$? We conjecture that the answer is $\le 1$ for any global field $F$ and finite group $G$. In the case when $F$ is a number field we show that the answer is always $\le 4[F:\Q]$. We show that assuming Schinzel's Hypothesis H the answer is always $\le 1$ if $F$ is a number field. We show unconditionally that the answer is always $\le 1$ if $F$ is a global function field.

We also show that for a broader class of fields $F$ than previously known, every finite group $G$ can be realized as the automorphism group of a finite extension $K/F$ (without restriction on the ramification).

An important new tool used in this work is a recent result of the author and C. Tsang, which says that for any finite group $G$ there exists a natural number $n$ and a subgroup $H\leqslant S_n$ of the symmetric group such that $N_{S_n}(H)/H\cong G$.\end{abstract}

\maketitle

\section{Introduction}

The Inverse Galois Problem (IGP) asks whether any finite group can be realized as the Galois group of some finite Galois extension $K/\Q$. It remains largely open, though many interesting cases have been solved, see e.g. \cite{MaMa18, Sha58, Zyw15, Zyw23}.
A refinement of this problem, called the Minimal Ramification Problem (MRP), asks for a given finite group $G$ what is the minimal number of ramified places of $\Q$ (by a place of $\Q$ we mean a finite prime or the infinite place) in a Galois extension $K/\Q$ with $\Gal(K/\Q)\cong G$. Let us denote this number by 
$r(G)$. If there is no Galois extension of $\Q$ with Galois group $G$ we set $r(G)=\infty$. It was conjectured by Boston and Markin \cite[Conjecture 1.2]{BoMa09} that for $G\neq 1$ we have $r(G)=\max(d(G^{\mathrm{ab}}),1)$, where $d(H)$ denotes the minimal number of generators of a group $H$. For progress on the MRP see \cite{BoMa09, Pla04, KNS10, BaSc20, BEF23}. The function field analogue of this problem was studied by the author, Bary-Soroker and Fehm \cite{BEF23}, by the author and Pirani \cite[\S 1.1]{EnPi24} and by Shusterman \cite{Shu24}.

In another direction, the following weak form of the IGP has been considered: given a finite group $G$, does there exist a finite extension $K/\Q$, not necessarily Galois, such that $\Aut(K/\Q)\cong G$? This question was answered affirmatively for every finite group $G$ by E. Fried and Koll\'ar \cite{FrKo78} (their proof contained a small gap later corrected by M. Fried \cite{Fri80}). A simpler alternative proof, using the Hilbert Irreducibility Theorem (HIT), was given by M. Fried \cite{Fri80}. A more elementary argument avoiding the use of HIT was given by Geyer \cite{Gey83}. The result has been generalized to all global base fields (in place of $\Q$) by Takahashi \cite{Tak80}, to all Hilbertian fields by Legrand and Paran \cite{LePa18} and to an even broader class of base fields by Deschamps and Legrand \cite{DeLe21}. In the present paper we give a new and shorter proof of these results, which applies to an even broader class of fields (see Theorem \ref{thm: not just hilbertian} and its proof below). For the corresponding problem for transcendental extensions, with possibly infinite automorphism groups, see \cite{DuGo87}.

The main focus of the present paper is the problem of realizing a given finite group $G$ as $\Aut(K/F)$ for some finite extension $K/F$ of a fixed global field $F$, such that the number of places of $F$ ramified in $K/F$ is as small as possible.

\begin{conj}\label{conj: main} Let $F$ be a number field and $G$ a finite group. There exists a finite extension $K/F$ such that $\Aut(K/F)\cong G$, ramified over at most a single place of $F$.\end{conj}

A conditional proof (assuming Schinzel's Hypothesis H) and partial results towards Conjecture \ref{conj: main} will be given in \S\ref{sec: number fields}. The question of whether one can find an everywhere unramified extension with $\Aut(K/F)=G$ is more difficult and will be discussed briefly in \S\ref{sec: unramified}.

We will show that our problem for a general finite group $G$ is closely related to the MRP in the special case of the symmetric and alternating groups. For a finite extension of global fields $K/F$ let us denote by $\Ram(K/F)$ the set of places of $F$ ramified in the extension and define
$$r_F(G)=\min_{K/F\,\mathrm{Galois}\atop{\Gal(K/F)\cong G}}|\Ram(K/F)|,$$ setting $r_F(G)=\infty$ if there is no Galois extension $K/F$ with $\Gal(K/F)=G$. Our main result is the following

\begin{thm}\label{thm: main} Let $F$ be a global field, $G$ a finite group. There exists a finite extension $K/F$ such that $\Aut(K)\cong G$, ramified over at most
$r=\underline{\lim}_{n\to\infty}\min(r_F(S_n),r_F(A_n))$ places of $F$.\end{thm}

We will use Theorem \ref{thm: main} to derive partial results and evidence towards Conjecture \ref{conj: main}.

\begin{thm}\label{cor: main} Let $F$ be a number field and $G$ a finite group. 
\begin{enumerate}\item[(i)] There exists a finite extension $K/F$ such that $\Aut(K/F)\cong G$, ramified over at most $4[F:\Q]$ places of $F$.
\item[(ii)] Assume that $r_\Q(S_n)= 1$ (this is a special case of the Boston-Markin conjecture mentioned above) for infinitely many $n$. Then Conjecture \ref{conj: main} holds for $F=\Q$.
\item[(iii)] Assume Schinzel's Hypothesis H (see \cite[p. 4]{Guy04} for its statement). Then Conjecture \ref{conj: main} holds.\end{enumerate}\end{thm}

In addition to Theorem \ref{thm: main}, a key ingredient in the proof of Theorem \ref{cor: main}(iii) is the following proposition, which generalizes \cite[Theorem 1.11]{BEF23} to arbitrary number fields and is of independent interest.

\begin{prop}\label{prop: schinzel} Assume Schinzel's Hypothesis H (in its usual form over the integers). Then $r_F(S_n)\le 1$ for any number field $F$.\end{prop}

Theorem \ref{thm: main} will be proved in \S\ref{sec: proof main} and Theorem \ref{cor: main} (as well as Proposition \ref{prop: schinzel}) will be deduced in \S \ref{sec: number fields}. The main ingredient in the proof of Theorem \ref{thm: main}, as well as Theorems \ref{thm: ff}, \ref{thm: not just hilbertian} below, is a recent result by the author and Tsang \cite{EnTs24_1} showing that every finite group is a normalizer quotient of arbitrarily large symmetric and alternating groups. See Proposition \ref{prop: group theory main} below for the precise statement.

\begin{rem} In the statements of Theorem \ref{thm: main}, Theorem \ref{cor: main} and Theorem \ref{thm: ff} below we may assert the existence of infinitely many non-isomorphic extensions $K/F$ with the required properties. This can be shown by augmenting the proofs with the argument used in the proof of Theorem \ref{thm: not just hilbertian} below to obtain infinitely many non-isomorphic extensions instead of just one. We omit this (straightforward) refinement to simplify the exposition.\end{rem}

\begin{rem} Our method can be used to handle other ramification conditions, as long as the corresponding result is available for infinitely many symmetric or alternating groups. For example, it was conjectured by Roberts and Venkatesh (\cite[Conjecture 1.1]{RoVe15} in the special case of the simple group $A_5$) that for infinitely many $n$ there exists a Galois extensions $K/\Q$ ramified only over $2,3,5,\infty$ with Galois group $S_n$ or $A_n$. Assuming this conjecture, one can produce (using the method of proof of Theorem \ref{thm: main}) finite extensions $K/\Q$ with $\Aut(K)=G$ and ramified only over $2,3,5,\infty$ for an arbitrary finite group $G$.\end{rem}

We also establish a strong function field analogue of Conjecture \ref{conj: main}.

\begin{thm}\label{thm: ff} Let $F$ be a global function field, $G$ a finite group and $P$ an arbitrary place of $F$. There exists a finite geometric extension $K/F$ such that $\Aut(K/F)\cong G$ and $K/F$ is unramified outside of $P$.\end{thm}

 The proof will be given in \S\ref{sec: function fields}.

\begin{rem} The extension $K/F$ is called geometric if $K$ and $F$ have the same fields of constants. This assumption appears in the function field analogue of the MRP studied in \cite{BEF23}, where it is imposed to make the behavior of the problem more in line with the case of $\Q$. 
\end{rem}

\subsection{Everywhere unramified extensions}
\label{sec: unramified}

One may ask under what conditions Conjecture \ref{conj: main} (resp. Theorem \ref{thm: ff}) can be strengthened to require $K/F$ to be everywhere unramified. This clearly cannot happen if $F=\Q$ if $G\neq 1$, because $\Q$ has no nontrivial unramified extensions. But it can happen for other number fields, with the answer depending on $G$ for some of them. This is clearly related to the question of whether $F$ has an infinite unramified extension, which is known to be difficult (see \cite{Mai00, Bri10, Kim16, Won11_} for a sample of special cases). 

In the function field case one may expect a more uniform answer. If $F$ has genus 0, it has no nontrivial geometric unramified extensions. If $F$ has genus 1 then all of its geometric unramified extensions come from isogenies, are all abelian, and the possible Galois groups are precisely the subgroups of the reduced class group $\mathrm{Cl}_0(F)$ (this follows from class theory or the theory of elliptic curves). In the case of genus $\ge 2$ we may pose the following

\begin{question} Let $F$ be a global function field of genus $g\ge 2$. Is it necessarily true that $r_F(A_n)=0$ for infinitely many $n$?\end{question}

If the answer is positive, by Theorem \ref{thm: main} there exists a finite unramified $K/F$ with $\Aut(K/F)\cong G$, for any finite group $G$. If we consider a univariate function field $F$ with an algebraically closed field of constants (instead of a finite field), a general result of Stevenson \cite[Corollary 3.5]{Ste96} guarantees a positive answer to the above question. The case of a finite field of constants appears to be open. We note that an unramified $S_n$-extension cannot exist if the reduced class group $\mathrm{Cl}_0(F)$ has odd order (this follows from class field theory).

\subsection{Extensions without restriction on ramification}

Legrand and Paran \cite[Theorem 1.1]{LePa18} have shown that for any Hilbertian field $F$ and finite group $G$, there exist infinitely many non-isomorphic finite extensions $K/F$ such that $\Aut(K/F)\cong G$. For the definition and basic properties of Hilbertian fields see \cite[\S 12]{FrJa08}. This result was extended to a broader class of fields by Deschamps and Legrand \cite[Th\'eor\`eme 4.1]{DeLe21}. We give a new and simpler proof of these results and extend them to a yet broader class of fields.

\begin{thm}\label{thm: not just hilbertian} Let $F$ be a field that has a Galois $S_n$-extension or a Galois $A_n$-extension for infinitely many $n$. Then for any finite group $G$ there exist infinitely many non-isomorphic finite extensions $K/F$ with $\Aut(K/F)\cong G$.\end{thm}

Theorem \ref{thm: not just hilbertian} will be proved in \S\ref{sec: proof main}.
The condition on $F$ in Theorem \ref{thm: not just hilbertian} is weaker than the assumptions in \cite[Th\'eor\`eme 4.1]{DeLe21}, which we do not repeat here in detail. We only mention that \cite[Th\'eor\`eme 4.1]{DeLe21} requires that $F$ has Galois $A_n$-extensions for infinitely many $n$, in addition to other restrictions. It is easy to construct fields satisfying the assumption of Theorem \ref{thm: not just hilbertian} but not this stronger condition. For example take $F$ to be the compositum of all Galois extensions $L/\Q$ such that $\Gal(L/\Q)$ has some $A_n$ as a quotient (then $F$ has Galois $S_n$-extensions for all $n$, but no Galois $A_n$-extensions). Finally we note that every Hilbertian field clearly satisfies the assumption of Theorem \ref{thm: not just hilbertian} because it is an elementary fact that a Hilbertian field has a Galois $S_n$-extension for every $n$ (just specialize the coefficients of the generic polynomial $x^n+a_{n-1}x^{n-1}+\ldots+a_0$ which has Galois group $S_n$), which gives another proof of \cite[Theorem 1.1]{LePa18}.
\\ \\
{\bf Acknowledgments.} The author would like to thank Elad Paran for encouraging the strengthening of Theorem \ref{cor: main}(iii) from a previous weaker version to its present form. The author would also like to thank Tammy Ziegler for an illuminating conversation pertaining to the possibility of sharpening Theorem \ref{cor: main}(i). The author was partially supported by Israel Science Foundation grant no. 2507/19.

\section{Proof of Theorems \ref{thm: main} and \ref{thm: not just hilbertian}}
\label{sec: proof main}

Let $F$ be a field, $G$ a finite group. To construct a finite extension $K/F$ with $\Aut(K/F)\cong G$, one begins with a finite Galois extension $L/F$ with $\Gal(L/F)=\Gamma$, where $\Gamma$ is an auxiliary group such that there exists a subgroup $H\leqslant \Gamma$ with $N_\Gamma(H)/H\cong G$ ($N_\Gamma(H)$ is the normalizer of $H$ in $\Gamma$). Then the fixed field $K=L^H$ is the requisite extension, by the following standard

\begin{prop}\label{prop: galois theory} Let $L/F$ be a finite Galois extension with $\Gal(L/F)=\Gamma$ and assume that there exists a subgroup $H\leqslant \Gamma$ such that $N_\Gamma(H)/H\cong G$. Then the fixed field $K=L^H$ satisfies $\Aut(K/L)\cong G$.\end{prop}

\begin{proof} This is a simple exercise in Galois theory.\end{proof}

If $F$ is a global field then $|\Ram(K/F)|\subset|\Ram(L/F)|$, so if one can control the ramification of $L/F$ then one also controls the ramification of $K/F$. In \cite{Fri80} the group $\Gamma$ is chosen as follows: first one embeds $G\leqslant S_n$ for some $n$ and then forms the permutational wreath product $A=S_N\wr_{S_n/G}S_n$ with some $N\ge 3$ (w.r.t. the natural action of $S_n$ on the coset space $S_n/G$). The stabilizer $H$ of all blocks and a point in the unit block (the block corresponding to the unit coset) satisfies $N_\Gamma(H)/H\cong G$. In \cite{Gey83} and \cite{LePa18} the special case $N=3$ is used\footnote{In all the cited sources the group $\Gamma$ is not described explicitly, but it is implicit as the Galois group of the largest extension considered.}. In \cite{DeLe21} $\Gamma$ is chosen to be an extension of $A_n$ by $S^N$, where $S$ is a finite non-abelian simple group. Unfortunately not much is known about the minimal ramification problem for such wreath products and extensions. Instead, we will work with $\Gamma=S_n$ and $\Gamma=A_n$. The main group-theoretic result we will need, recently obtained by the author jointly with C. Tsang, is the following

\begin{prop}\label{prop: group theory main} Let $G$ be a finite group. For all sufficiently large $n\ge n_0(G)$ and $\Gamma\in\{S_n,A_n\}$, there exists a subgroup $H\leqslant \Gamma$ such that $N_\Gamma(H)/H\cong G$.\end{prop}

\begin{proof} This is \cite[Corollary 1.2]{EnTs24_1}.\end{proof}

The proof of Proposition \ref{prop: group theory main} uses a result of Cornulier and Sambale \cite{Sam24_1} on the realization of any finite group as the outer automorphism group of another finite group, as well as the determination of the normalizer of the inner holomorph of a finite centerless indecomposable group. See \cite{EnTs24_1} for further details.

\begin{proof}[Proof of Theorem \ref{thm: main}] 
    Assume that $F$ is a global field. Let $n_0(G)$ be as in Proposition \ref{prop: group theory main} and $$r=\underline\lim_{n\to\infty}\min(r_F(S_n),r_F(A_n)).$$ Pick $n\ge n_0(G)$ and $\Gamma\in\{S_n,A_n\}$ such that there exists a Galois extension $L/F$ with $\Gal(L/F)=\Gamma$ ramified over $r$ places of $F$ (this is possible by the definition of $r$). By Proposition \ref{prop: group theory main} there exists $H\leqslant \Gamma$ with $N_\Gamma(H)/H\cong G$. By Proposition \ref{prop: galois theory} the subextension $K=L^H$ satisfies $\Aut(K/F)\cong G$, and since $\Ram(K/F)\subset\Ram(L/F)$ it is ramified over at most $r$ places of $F$, as required.
\end{proof}

\begin{proof}[Proof of Theorem \ref{thm: not just hilbertian}]

 Let $F$ be a field as in Theorem \ref{thm: not just hilbertian}, $G$ a finite group. If $G=1$ the assertion is trivial and we will treat the case $|G|=2$ separately below. For now assume $|G|> 2$. By assumption there exists an infinite increasing sequence of natural numbers $n$ and Galois extensions $L_n/F$ with $\Gal(L_n/F)\cong\Gamma_n,\,\Gamma_n\in\{S_n,A_n\}$. By Proposition \ref{prop: group theory main} we may assume (looking only at $n\ge n_0(G)$) that there exist subgroups $H_n\leqslant \Gamma_n$ such that $N_{\Gamma_n}(H_n)/H_n\cong G$. Denoting $K_n=L_n^{H_n}$ we have (by Proposition \ref{prop: galois theory}) $\Aut(K_n/F)\cong G$. It remains to show that infinitely many of the $K_n$ are non-isomorphic extensions of $F$.

 If $\Gamma_n=A_n$ for infinitely many $n$, then $L_n/F$ are pairwise linearly disjoint for $n\ge 5$ (because the groups $A_n$ are have no common quotient) and therefore so are the $K_n/F$. A fortiori they are non-isomorphic.

 It remains (after removing finitely many $n$) to treat the case that $\Gamma_n=S_n$ for all $n$ in our sequence. In this case either infinitely many of the $L_n/F$ are pairwise linearly disjoint (and then we conclude as above), or we may assume (after removing finitely many $n$) that all the $L_n/F$ contain a common quadratic extension $E/F$, over which they are pairwise linearly disjoint (here we use the fact that the only nontrivial quotient of $S_n$ for $n\ge 5$ is $S_n/A_n\cong C_2$). In that case either infinitely many of the $K_n/F$ are non-isomorphic (in which case we are done), or we may assume (after removing finitely many $n$) that $K_n=E$ for all $n$, in which case $|G|=2$ contrary to assumption.

 It remains to treat the case $|G|=2$. If there is a Galois $A_n$-extension of $F$ for infinitely many $n$ the argument above works for this case as well. Otherwise there is a Galois $S_n$-extension $L_n/F$ for an infinite sequence of $n$. Set $H_n=A_{n-1}\leqslant S_n$ and $K_n=L_n^{H_n}$. It is easy to check that $N_{S_{n}}(H_n)/H_n$ is cyclic of order 2 for $n\ge 5$, so by Proposition \ref{prop: galois theory} we have $\Aut(K_n/F)\cong G$ and the extensions $K_n/F$ are non-isomorphic (since they are not contained in the unique quadratic subextensions $F\subset E_n\subset L_n$).

\end{proof}

\section{Number fields: proof of Theorem \ref{cor: main}}
\label{sec: number fields}

\begin{proof}[Proof of Theorem \ref{cor: main}] Recall that we denote $r(G)=r_\Q(G)$.\\
{\bf (i).} A result of Bary-Soroker and Schlank \cite[Theorem 1.5]{BaSc20} asserts that for any $n\ge 2$ there exist infinitely many pairwise linearly disjoint Galois extensions $L/\Q$ with $\Gal(L/\Q)\cong S_n$ ramified over 4 places of $\Q$ \footnote{The original statement asserts the existence of one such extension, but the method produces infinitely many linearly disjoint ones.}. One of them is linearly disjoint from $F/\Q$. For such an extension, $LF/F$ is Galois with $\Gal(LF/F)\cong\Gal(L/\Q)\cong S_n$ and $LF/F$ can only be ramified over the places of $F$ lying over $\Ram(L/\Q)$. The number of such places is $\le[F:\Q]|\Ram(L/\Q)|=4[F:\Q]$. Consequently $r_F(S_n)\le 4[F:\Q]$  for all $n$ and the assertion follows from Theorem \ref{thm: main}.\\
{\bf (ii).} Immediate from Theorem \ref{thm: main}.\\
{\bf (iii).} Immediate from Theorem \ref{thm: main} and Proposition \ref{prop: schinzel}, which will be proved in \S\ref{sec: schinzel} below.
\end{proof}

\begin{rem} It may be possible to improve the bound $4[F:\Q]$ in Theorem \ref{cor: main}(i) to $[F:\Q]+3$ by performing the construction of \cite[Theorem 1.5]{BaSc20} directly over $F$ (as opposed to doing it over $\Q$ and base-changing to $F$ as we did) and then applying Theorem \ref{thm: main}. This would require generalizing a result of Green-Tao \cite[Corollary 1.7]{GrTa10} to number fields, which should be doable but would require significant technical work. An alternative approach would be to strengthen  \cite[Corollary 1.7]{GrTa10} to apply to primes in prescribed Chebotarev classes (satisfying certain mild conditions). This may be achievable by combining the Green-Tao machinery with the results of Matthiesen \cite{Mat18}.
Either approach would require extensive work and we have not attempted to pursue the above improvement.
    
\end{rem}

\subsection{An alternative approach to Theorem \ref{cor: main}(i)}

Theorem \ref{cor: main}(i), as well as Theorem \ref{thm: ff}, can be established by an alternative method which does not use the group-theoretic result of Proposition \ref{prop: group theory main}, instead relying on a group-theoretic result of Frucht \cite{Fru39} which says that every finite group can be realized as the automorphism group of a finite simple graph. This is similar to the approach of Fried-Koll\'ar \cite{FrKo78}. Let us sketch the construction for proving Theorem \ref{cor: main}(i) in the case $F=\Q$. Let $G$ be a finite group and $\Delta=(V,E)$ a finite simple graph with $\Aut(\Delta)\cong G$. Set $n=|V|$ and assume $V=\{1,\ldots,n\}$. Let $p,q,r$ be primes such that
\begin{enumerate}
\item[(i)] $r=n^np+(n-1)^{n-1}q$.
\item[(ii)] $f=X^n-X^{n-1}-\frac pq$ satisfies $\Gal(f/\Q)=S_n$.
\end{enumerate}
It follows from \cite[Proof of Theorem 1.5]{BaSc20} that such triples exist with $\min(p,q,r)$ arbitrarily large. Let $\alpha_1,\ldots,\alpha_n\in\overline\Q$ be the roots of $f$ and for each edge $\{i,j\}\in E$ pick $\beta_{ij}\in\overline\Q$ such that $\beta_{ij}^r=(\alpha_i-\alpha_j)^2$. Set $K=\Q(\alpha_1,\ldots,\alpha_n,\beta_{ij}:\,\{i,j\}\in E)$. A simple calculation shows that $\Ram(K/\Q)=\{p,q,r,\infty\}$. One can show by arguments similar to \cite{FrKo78} that $\Aut(K/\Q)\cong\Aut(\Delta)\cong G$. Here a crucial fact is that $\alpha_i-\alpha_j,\,1\le i<j\le n$ are multiplicatively independent, which can be deduced from the fact that the inertia group of $\Q(\alpha_1,\ldots,\alpha_n)/\Q$ over $r$ is generated by a transposition. Variations of this construction can prove the full Theorem \ref{cor: main}(i) and Theorem \ref{thm: ff}, but we do not include the details here since we derive these results from Theorem \ref{thm: main} anyway.

\subsection{Proof of Proposition \ref{prop: schinzel}}
\label{sec: schinzel}

Assume Schinzel's Hypothesis H (henceforth referred to as Hypothesis H). Let $F$ be a fixed number field, $n\ge 2$ a natural number (the case $n=1$ is trivial). We will show that $r_F(S_n)\le 1$, establishing Proposition \ref{prop: schinzel}. In addition to Hypothesis H, a key tool in the proof is Hilbert's Irreducibility Theorem (HIT) in various forms. We will assume basic familiarity with the subject, as presented e.g. in \cite[\S 12-13]{FrJa08}. We will need the following variant of the HIT:

\begin{prop}\label{prop: hit}
Let $\mathcal H\subset\Q$ be a Hilbert subset in the sense of \cite[\S 12]{FrJa08}. There exists $\delta=\delta(\mathcal H)>0,\,k=k(\mathcal H)\in\N$ with the following property: for any $C>0$ there exists a squarefree $v\in\N,\,v>C$ composed of $\le k$ prime factors and a subset $U\subset(\Z/v\Z)^\times$ of size $|U|\ge\delta v$, such that $t\in\mathcal H$ for any $t\in\Z$ with $(t\bmod v)\in U$.\end{prop}

\begin{proof} The proposition is a refinement of \cite[Theorem 13.3.5(a)]{FrJa08} in the special case $K=\Q$ and a careful tracking of the original proof shows that it essentially establishes this refinement. We omit the details for brevity, noting that this is a fairly standard argument based on the Chebotarev Density Theorem in function fields.\end{proof}

Let $\sigma_1,\ldots,\sigma_m:F\to\R$ be the real embeddings of $F$ and denote by $\mathcal O_F$ the ring of integers of $F$. Denote also $d=[F:\Q]$, let $\alpha\in\mathcal O_F$ be a primitive element of $F/\Q$ and $\alpha_1=\alpha,\,\alpha_2,\ldots,\alpha_d$ its conjugates over $\Q$.

Let $A_1,\ldots,A_n$ be independent variables. We will use the shorthand $\mathbf A=(A_1,\ldots,A_n)$.  Consider the generic monic discriminant
$$D(\mathbf A)=D(A_1,\ldots,A_n)=\disc_X(X^n+A_1X^{n-1}+\ldots+A_n).$$ Denote \begin{equation}\label{eq: prod X-i}\prod_{i=1}^n(X-i)=X^n+\sum_{i=1}^nc_iX^{n-i}\in\Z[X]\end{equation} and set $$P=\prod_{p\le dn(n-1)\atop{\mathrm{prime}}}p.$$

Set
\begin{equation}\label{eq: fxat}f(X;\mathbf A,T)=X^n+\sum_{i=1}^nc_i(T^i+A_i)X^{n-i}+P\alpha A_n\in\mathcal O_F[\mathbf A,T][X].\end{equation}
Substituting $T=0$ we see that $\Gal(f(X;\mathbf A,0)/\overline{\mathbb Q}(\mathbf A))\cong S_n$ and therefore $\Gal(f(X;\mathbf A,T)/\Q(\mathbf A,T))\cong S_n$ (since the Galois group can only shrink under specialization and extension of the base field and both groups are contained in $S_n$).
We aim to show (assuming Hypothesis H) that for a suitable choice of $\mathbf a\in\Z^n,\,t\in\Z$ the specialized polynomial $f(X;\mathbf a,t)\in\mathcal O_F[X]$ has Galois group $S_n$ over $F$ and its splitting field over $F$ is ramified over at most one place. This would establish $r_F(S_n)\le 1$.

Denote \begin{equation}\label{eq: def D_i}D_j(\mathbf A,T)=D\left(c_1(A_1+T),\ldots,c_{n-1}(A_{n-1}+T^{n-1}),c_n(A_n+T^n)+P\alpha_jA_n\right)\in\Q(\alpha_j)[\mathbf A, T],\quad 1\le j\le d.\end{equation} It is well-known that the generic discriminant $D(\mathbf A)$ is absolutely irreducible (i.e. irreducible in $\overline\Q[\mathbf A]$) and consequently $D_j(\mathbf A,T)\in\overline\Q[\mathbf A,T]$ are irreducible (they are obtained from $D$ by an invertible change of variables). It is easy to see that for a suitable choice of primitive element $\alpha$ (which we henceforth assume to have been made) the polynomials $D_1,\ldots,D_d$ are distinct and therefore (e.g. using (\ref{eq: Dj expansion}) below) non-associate. Since $\Gal(\overline \Q/\Q)$ acts transitively on $\alpha_1,\ldots,\alpha_d$ (and hence on $D_1,\ldots,D_d$) it follows that 
\begin{equation}\label{eq: def H}H(\mathbf A,T)=\prod_{j=1}^dD_j(\mathbf A,T)\in\Z[\mathbf A,T]\end{equation} is irreducible in $\Q[\mathbf A,T]$. We note that (by (\ref{eq: fxat}),(\ref{eq: def D_i}),(\ref{eq: def H})) for any $\mathbf a\in\Z^n,\,t\in\Z$ we have
\begin{equation}\label{eq: H disc}H(\mathbf a,t)=N_{F/\Q}(\disc(f(X;\mathbf a,t))).\end{equation} By \cite[Theorem 13.3.5(c) and Corollary 12.2.3]{FrJa08} we can pick $\mathbf a\in\Z^n$ such that
\begin{enumerate}
    \item[(i)]
$H(\mathbf a,T)\in\Z[T]$ is irreducible in $\Q[T]$.
\item[(ii)] $\Gal(f(X;\mathbf a,T)/\Q(T))\cong S_n$. 
\item[(iii)]The reduction modulo $p$ of $X^n+\sum_{i=1}^nc_ia_iX^{n-i}$ is separable for any prime $p|P$
\end{enumerate}
(to satisfy (iii) we just need to pick each $a_i$ in a suitable arithmetic progression modulo $P$).

Fix a choice of $\mathbf a$ satisfying (i-iii). We claim that it automatically satisfies the following strengthening of (i):
\begin{enumerate}\item[(iv)]
$H(\mathbf a,T)$ is irreducible in $\Z[T]$.\end{enumerate}
Indeed it is easy to see from (\ref{eq: def D_i}),(\ref{eq: prod X-i}) and the fact that $D(\mathbf A)$ is homogeneous of degree $n(n-1)$ with respect to the weights $w(A_i)=i$ \cite[(12.1.24)]{GKZ94} that \begin{equation}\label{eq: Dj expansion}D_j(\mathbf A,T)=\left(\prod_{1\le i<k\le n}(k-i)^2\right)T^{n(n-1)}+\mbox{lower degrees in }T\end{equation}
and therefore (using (\ref{eq: def H}))
$$H(\mathbf a,T)=\left(\prod_{1\le i<k\le n}(k-i)^{2d}\right)T^{dn(n-1)}+\mbox{lower degrees in }T.$$
It follows that for a prime $p>n$ we have $p\nmid H(\mathbf a,T)$. If $p\le n$ then $p|P$, and substituting $T=0$ in (\ref{eq: fxat}) and using (iii) and that $H(\mathbf a,0)=N_{F/\Q}(\disc_t(f(X;\mathbf a,0)))$ by (\ref{eq: H disc}), we see that $p\nmid H(\mathbf a,0)$ and therefore $p\nmid H(\mathbf a,T)$ in this case as well. Thus $H(\mathbf a,T)$ is content-free (not divisible by any prime) and (i) implies (iv).

By Proposition \ref{prop: hit} there exist $v,u\in\N$ such that 
\begin{enumerate}
    \item[(v)] $(u,v)=1$.
    \item[(vi)]For any $t\in\mathbb Z$ with $t\equiv u\pmod v$ we have $\Gal(f(X;\mathbf a,t)/\Q)\cong S_n$.
    \item[(vii)] $H(\mathbf a,u)\not\equiv 0\pmod v$.
    \end{enumerate}
A suitable $u$ can be found for a sufficiently large $v$, because for the $v$ given by Proposition \ref{prop: hit} at least $\delta v$ residues $(u\bmod v)\in(\Z/v\Z)^\times$ satisfy (vi) for some $\delta>0$ independent of $v$, and the number of roots of $H(\mathbf a,T)$ modulo $v$ is bounded by a constant independent of $v$, since the latter has a bounded number of prime factors (here we are using that $H(\mathbf a,T)$ is nonzero modulo any prime by (iv)).

Next we want to use Hypothesis H to pick a large $s\in\N$ such that $H(\mathbf a,u+vs)$ is prime. We need to show that there is no local obstruction, i.e. for any (rational) prime $p$ we may pick $s\in\mathbb Z$ such that $p\nmid H(\mathbf a,u+vs)$. If $p|P$ then by (\ref{eq: fxat}),(\ref{eq: H disc}) and assumption (iii) above we have that
$$H(\mathbf a,0)=N_{F/\Q}(\disc_X(f(X;\mathbf a,0)))\not\equiv 0\pmod p.$$
If $p\nmid P$, then once again using the fact that $D(\mathbf A)$ has degree $2n-2$ with respect to the weights $w(A_i)=i$ we obtain
$p> dn(n-1)\ge\deg_T(D_1(\mathbf a,T)\cdots D_d(\mathbf a,T))=\deg_TH(\mathbf a,T)$.  In the case $p\nmid v$ we see that $H(\mathbf a,u+vs)$ has at most $\deg_TH(\mathbf a,T)<p$ roots modulo $p$, so some $s$ is not a root.
In the case $p|v$, by assumption (vii) above we may take $s=0$.
We have established that there is no local obstruction to $H(\mathbf a,u+vs)$ being prime and since $H(\mathbf a,u+vX)$ is irreducible in $\mathbb Q[X]$ (because $H(\mathbf a,X)$ is by (i)), Hypothesis H implies that there are arbitrarily large $t=u+vs$ with the following properties:
\begin{enumerate}
\item[(viii)] $H(\mathbf a,t)$ is prime.
\item[(ix)] $t\equiv u\pmod v$.
\end{enumerate}

We claim that for sufficiently large $t$ satisfying (viii-ix) the polynomial $f(X;\mathbf a,t)$ satisfies the required properties. First observe that by (vi) above we have $\Gal(f(X;\mathbf a,t)/\Q)\cong S_n$. Next since $H(\mathbf a,t)=N_{F/\Q}(\disc_X(f(X;\mathbf a,t)))$ is prime, we have that $\mathfrak q=(\disc_Xf(X;\mathbf a,t))$ is a prime ideal of $\mathcal O_F$ and is therefore the only possible finite place over which the splitting field of $f(X;\mathbf a,t)$ over $F$ ramifies. Finally if for each embedding $\sigma_j:F\to\R$ and polynomial $g=\sum{g_i}X^i\in F[X]$ we denote $g^{\sigma_j}=\sum\sigma_j(g_i)X^i\in\R[X]$, we see from (\ref{eq: prod X-i}) and (\ref{eq: fxat}) that the roots of $f(tX;\mathbf a,t)^{\sigma_j}$ approach $1,2,\ldots,n$ as $t\to\infty$ and in particular have to be real for $t$ sufficiently large (since they cannot occur in conjugate pairs). Consequently the splitting field of $f(X;\mathbf a,t)$ over $F$ is unramified over the infinite places of $F$ and is therefore ramified only over $\mathfrak q$. This concludes the proof that $r_F(S_n)\le 1$ and establishes Proposition \ref{prop: schinzel}.

\section{Function fields: proof of Theorem \ref{thm: ff}}
\label{sec: function fields}

Let $F$ be a global function field with field of constants $\F_q$ and characteristic $p$, $P$ a place of $F$.
A simple consequence of the Riemann-Roch theorem is that there exists $T\in F\setminus\F_q$ that has a pole only at $P$ (see e.g. \cite[Theorem 6.13]{Ros02} for a more general statement). 

If $p>2$, by \cite[Theorem 1]{EnPi24} for any $n\ge p$ there exists a Galois extension $M_n/\F_q(T)$ with $\Gal(M_n/\F_q(T))\\=A_n$, ramified only over the place $T=\infty$. In the case $p=2$ such an extension exists for any $9\le n\equiv 1\pmod 8$ by \cite[Theorem 2]{Abh93} and \cite[Equation (5.27)]{AOS94} (one takes the splitting field of $f(X)=X^n+TX^{n-4}+1$). In what follows we only consider $n$ such that $\max(p,9)\le n\equiv 1\pmod 8$.

The extension $M_n/\F_q(T)$ is necessarily geometric (since $A_n$ has no cyclic quotients). The extensions $M_n/\F_q(T)$ are all pairwise linearly disjoint since the groups $A_n$ are simple and non-isomorphic. Therefore we may find $n\ge n_0(G)$ such that $M=M_n$ is linearly disjoint from $F$ over $\F_q(T)$. Here $n_0(G)$ is chosen to satisfy the assertion of Proposition \ref{prop: group theory main}.

Let $L$ be the compositum of $F$ and $M$ over $\F_q(T)$ (inside a common algebraic closure). Since $F,M$ are linearly disjoint over $\F_q(T)$ we have $\Gal(L/F)\cong\Gal(M/\F_q(T))\cong A_n$. The extension $L/F$ is geometric because $A_n$ has no cyclic quotients. The extension $L/F$ can only ramify over a place of $F$ that lies over the place $T=\infty$ of $\F_q(T)$ (since $T=\infty$ is the only ramified place of $M/\F_q(T)$). By assumption the only pole of $T$ is $P$, so $P$ is the only place lying over $T=\infty$ and $L/F$ is unramified outside of $P$. 

Denote $\Gamma=\Gal(L/F)\cong A_n$. By Proposition \ref{prop: group theory main} and the assumption $n\ge n_0(G)$, there exists $H\leqslant \Gamma$ such that $N_\Gamma(H)/H\cong G$. By Proposition \ref{prop: galois theory} there exists a subextension $F\subset K\subset L$ with $\Aut(K/F)\cong G$. It is geometric and unramified outside of $P$ (since $L/F$ is), so it satisfies the requirements of Theorem \ref{thm: ff}. This completes the proof of Theorem \ref{thm: ff}.

\bibliography{mybib}
\bibliographystyle{alpha}

\end{document}